\documentclass[reqno]{amsart}

\usepackage{amsmath,amssymb,amscd,amsthm,amsfonts,amstext,amsbsy,mathrsfs,xfrac,upgreek,mathtools,stmaryrd,enumitem,bbm,bbold}

\usepackage{hyperref}

\newcommand{\KK}{\mathrm{K}}

\newcommand{\HHH}[1]{\mathrm{H}({#1})}

\newcommand{\seq}[2]{\langle #1 \mid #2 \rangle}

\newcommand{\anf}[1]{{\text{``}\hspace{0.3ex}{#1}\hspace{0.3ex}\text{''}}}

\newcommand{\id}{\operatorname{id}}
\newcommand{\cof}[1]{\operatorname{cof}{({#1})}}
\newcommand{\otp}[1]{{{\rm{otp}}\left(#1\right)}}

\newcommand{\On}{{\mathrm{On}}}
\newcommand{\Lim}{{\mathrm{Lim}}}

\newcommand{\ZFC}{{\sf ZFC}}

\newcommand{\map}[3]{{#1}:{#2}\longrightarrow{#3}}

\newcommand{\Set}[2]{\{{#1}~\vert~{#2}\}}
\newcommand{\ran}[1]{{{\rm{ran}}(#1)}}

\newcommand{\dom}[1]{{{\rm{dom}}(#1)}}

\newcommand{\meas}[1]{{{\rm{meas}}(#1)}}
\newcommand{\lp}[1]{{{\rm{lp}}(#1)}}
\newcommand{\VV}{{\rm{V}}}

\newcommand{\JJ}{{\rm{J}}}
\newcommand{\LL}{{\rm{L}}}

\newcommand{\betrag}[1]{\vert{#1}\vert}
\newcommand{\crit}[1]{{{\rm{crit}}(#1)}}
\newcommand{\rank}[2]{{\rm{rnk}}_{{#2}}({#1})}

\newcommand{\tc}[1]{{\rm{tc}}({#1})}

\newcommand{\POT}[1]{{\mathcal{P}}({#1})}
\newcommand{\Ult}[2]{{\rm{Ult}}({#1},{#2})}

\newcommand{\GCH}{{\rm{GCH}}}

\newtheorem{theorem}{Theorem}[section]
\newtheorem{lemma}[theorem]{Lemma}
\newtheorem{corollary}[theorem]{Corollary}

\newtheorem{question}{Question}

\newtheorem*{claim*}{Claim}

\newtheorem*{subclaim*}{Subclaim}

\theoremstyle{definition}
\newtheorem{definition}[theorem]{Definition}

\theoremstyle{remark}

\newenvironment{enumerate-(a)}{\begin{enumerate}[label={\upshape (\alph*)}, leftmargin=2pc]}{\end{enumerate}}

\newenvironment{enumerate-(a)-r}{\begin{enumerate}[label={\upshape (\alph*)}, leftmargin=2pc,resume]}{\end{enumerate}}

\newenvironment{enumerate-(A)}{\begin{enumerate}[label={\upshape (\Alph*)}, leftmargin=2pc]}{\end{enumerate}}

\newenvironment{enumerate-(A)-r}{\begin{enumerate}[label={\upshape (\Alph*)}, leftmargin=2pc,resume]}{\end{enumerate}}

\newenvironment{enumerate-(i)}{\begin{enumerate}[label={\upshape (\roman*)}, leftmargin=2pc]}{\end{enumerate}}

\newenvironment{enumerate-(i)-r}{\begin{enumerate}[label={\upshape (\roman*)}, leftmargin=2pc,resume]}{\end{enumerate}}

\newenvironment{enumerate-(I)}{\begin{enumerate}[label={\upshape (\Roman*)}, leftmargin=2pc]}{\end{enumerate}}

\newenvironment{enumerate-(I)-r}{\begin{enumerate}[label={\upshape (\Roman*)}, leftmargin=2pc,resume]}{\end{enumerate}}

\newenvironment{enumerate-(1)}{\begin{enumerate}[label={\upshape (\arabic*)}, leftmargin=2pc]}{\end{enumerate}}

\newenvironment{enumerate-(1)-r}{\begin{enumerate}[label={\upshape (\arabic*)}, leftmargin=2pc,resume]}{\end{enumerate}}

\begin{document}

\author{Philipp L\"ucke} 
\address{Philipp L\"ucke, Mathematisches Institut, Universit\"at Bonn,
Endenicher Allee 60, 53115 Bonn, Germany}
\email{pluecke@math.uni-bonn.de}
\urladdr{}

\author{Philipp Schlicht}
\address{Philipp Schlicht, Mathematisches Institut, Universit\"at Bonn,
Endenicher Allee 60, 53115 Bonn, Germany}
\email{schlicht@math.uni-bonn.de}
\urladdr{}



\thanks{During the preparation of this paper, both authors were partially supported by DFG-grant LU2020/1-1. 
The main results of this paper were obtained while the authors were participating in the \emph{Intensive Research Program on Large Cardinals and Strong Logics} at the Centre de Recerca Matem\`{a}tica in Barcelona during the fall of 2016. The authors would like to thank the organizers for the opportunity to participate in the program. 
Moreover, the authors would like to thank Peter Koepke and Philip Welch for helpful discussions on the topic of the paper.}

\title[Measurable cardinals and good $\Sigma_1(\kappa)$-wellorderings]{Measurable cardinals and good $\mathbf{\Sigma_1({\boldsymbol \kappa})}$-wellorderings}

\begin{abstract} 
 We study the influence of the existence of large cardinals on the existence of wellorderings of power sets of infinite cardinals $\kappa$ with the property that the collection of all initial segments of the wellordering is definable by a $\Sigma_1$-formula with parameter $\kappa$. 
A short argument shows that the existence of a measurable cardinal $\delta$ implies that such wellorderings do not exist at $\delta$-inaccessible cardinals of cofinality not equal to $\delta$ and their successors.  
In contrast, our main result shows that these wellorderings exist at all other uncountable cardinals in the minimal model containing a measurable cardinal. In addition, we show that measurability is the smallest large cardinal property that interferes with the existence of such wellorderings at uncountable cardinals and we generalize the above result to the minimal model containing two measurable cardinals. 
\end{abstract}

\maketitle


\section{Introduction}\label{section:intro}

We study the interplay between the existence of large cardinals and the existence of very simply definable wellorders of power sets of certain infinite cardinals. In this paper, we focus on the following type of definable wellorders:

\begin{definition}
 Fix sets $y_0,\ldots,y_{n-1}$. 
 \begin{enumerate}
  \item A set $X$ is \emph{$\Sigma_1(y_0,\ldots,y_{n-1})$-definable} if there is a $\Sigma_1$-formula $\varphi(v_0,\ldots,v_n)$ satisfying $$X ~ = ~ \Set{x}{\varphi(x,y_0,\ldots,y_{n-1})}.$$

  \item A wellordering $\lhd$ of a set $X$ is a \emph{good $\Sigma_1(y_0,\ldots,y_{n-1})$-wellordering of $X$} if the set $$I(\lhd) ~ = ~ \Set{\Set{x\in X}{x\lhd y}}{y\in X}$$ of all proper initial segments of $\lhd$ is $\Sigma_1(y_0,\ldots,y_{n-1})$-definable. 
 \end{enumerate}  
\end{definition}

Note that, if $\lhd$ is a good $\Sigma_1(y_0,\ldots,y_{n-1})$-wellordering of a $\Sigma_1(y_0,\ldots,y_{n-1})$-definable set $X$, then $\lhd$ is also $\Sigma_1(y_0,\ldots,y_{n-1})$-definable, because $\lhd$ consists of all pairs $\langle x,y\rangle$ in $X\times X$ with the property that there is an $A\in I(\lhd)$ with $x\in A$ and $y\notin A$. Moreover, given a good $\Sigma_1(y)$-wellordering $\lhd$, the statement \anf{\emph{$x$ is the $\alpha$-th element of $\lhd$}} can be expressed by a $\Sigma_1$-statement with parameters $\alpha$, $x$ and $y$. In particular, if $\kappa$ is an infinite cardinal and $\lhd$ is a good $\Sigma_1(y)$-wellordering of $\POT{\kappa}$ with $y\in\HHH{\kappa^+}$, then the $\Sigma_1$-Reflection Principle implies that $\lhd$ has order-type $\kappa^+$ and hence the $\GCH$ holds at $\kappa$. Classical results of G\"odel show that in his constructible universe $\LL$, there is a good $\Sigma_1(\kappa)$-wellordering of $\POT{\kappa}$ for every infinite cardinal $\kappa$. In contrast, the results of \cite{Lightface} show that stronger large cardinal axioms imply the nonexistence of such wellorderings for certain cardinals. For example, a combination of {\cite[Theorem 1.7]{Lightface}} with {\cite[Lemma 5.5]{Lightface}} shows that there is no good $\Sigma_1(\kappa)$-wellordering of $\POT{\kappa}$ if $\kappa$ is a measurable cardinal with the property that there two different normal ultrafilters on $\kappa$ and {\cite[Theorem 1.2]{Lightface}} states that no wellordering of $\POT{\omega_1}$ is $\Sigma_1(\omega_1)$-definable if there is a measurable cardinal above a Woodin cardinal.

In this paper, we study the influence of measurable cardinals on the existence of good $\Sigma_1$-wellorderings. The following two lemmas provide examples of such implications. The results of this paper will show that it is possible that a measurable cardinal exists and good $\Sigma_1$-wellorders exists at all cardinals that are not ruled out by these two lemmas.

\begin{lemma}\label{lemma:MeasurableWOReals}
 If there is a measurable cardinal and $x\in\HHH{\omega_1}$, then no wellordering of $\POT{\omega}$ is $\Sigma_1(x)$-definable.  
\end{lemma}

\begin{proof}
 Assume, towards a contradiction, that there is a $\Sigma_1(x)$-definable wellordering of $\POT{\omega}$ for some $x\in\HHH{\omega_1}$. By {\cite[Lemma 25.25]{MR1940513}}, this assumptions implies that there is a $\mathbf{\Sigma}^1_2$-wellordering of the reals and hence there is $\mathbf{\Sigma}^1_2$-subset of the reals without the Baire property. Classical results of Solovay (see {\cite[Corollary 14.3]{MR1994835}}) show that this conclusion implies that there are no measurable cardinals. 
\end{proof}

The following results show that the existence of a measurable cardinal also impose restrictions on the existence of good $\Sigma_1$-wellorderings at many cardinals above the measurable cardinal.  Remember that, given cardinals $\delta<\kappa$, we say that \emph{$\kappa$ is $\delta$-inaccessible} if $\lambda^\delta<\kappa$ holds for all $\lambda<\kappa$.

\begin{lemma}\label{lemma:MeasurableWOabove}
 Let $\delta$ be a measurable cardinal and let $\nu>\delta$ be a $\delta$-inaccessible cardinal with $\cof{\nu}\neq\delta$. If $\kappa$ is a cardinal with $\nu\leq\kappa\leq\nu^+$, then there is no good $\Sigma_1(\kappa)$-wellordering of $\POT{\kappa}$. 
\end{lemma}

\begin{proof}
 Assume, towards a contradiction, that there is a good $\Sigma_1(\kappa)$-wellordering $\lhd$ of $\POT{\kappa}$. Pick a $\Sigma_1$-formula $\varphi(v_0,v_1)$ with $I(\lhd)=\Set{A}{\varphi(A,\kappa)}$.  Let $U$ be a normal ultrafilter on $\delta$ and let $\map{j_U}{\VV}{M=\Ult{\VV}{U}}$ denote the induced ultrapower map. 

 \begin{claim*}
  $j_U(\nu)=\nu$.
 \end{claim*}

 \begin{proof}[Proof of the Claim]
  First, assume that $\cof{\nu}>\delta$. Pick $\map{f}{\delta}{\nu}$. Then there is $\lambda<\nu$ with $\ran{f}\subseteq\lambda$. Since our assumptions imply that the set of all functions from $\delta$ to $\lambda$ has cardinality less than $\nu$ and all elements of the ordinal $j_U(f)(\delta)$ are of the form $j_U(g)(\delta)$ for some function $\map{g}{\delta}{\lambda}$, we can conclude that $j_U(f)(\delta)<\nu$. This argument shows that $j_U(\nu)=\nu$ holds in this case, because every element of $j_U(\nu)$ is of the form $j_U(f)(\delta)$ for some $\map{f}{\delta}{\nu}$. 

  Now, assume that $\cof{\nu}<\delta$ and fix a cofinal sequence $\seq{\nu_\xi<\nu}{\xi<\cof{\nu}}$ in $\nu$. Pick $\map{f}{\delta}{\nu}$. In this situation, the normality of $U$ yields $\xi_*<\cof{\nu}$ with $f^{{-}1}[\nu_{\xi_*}]\in U$. In particular, every element of the ordinal $j_U(f)(\delta)$ is of the form $j_U(g)(\delta)$ for some function $\map{g}{\delta}{\nu_{\xi_*}}$. Since our assumptions imply that the set of all functions from $\delta$ to $\nu_{\xi_*}$ has cardinality less than $\nu$, this shows that the ordinal $j_U(f)(\delta)$ has cardinality less than $\nu$. As above, we can conclude that $j_U(\nu)=\nu$. 
 \end{proof}

 \begin{claim*}
  $j_U(\kappa)=\kappa$.
 \end{claim*}

 \begin{proof}
   By the above claim, we may assume that $\kappa=\nu^+$. Then the above claim implies that $j_U(\nu)<\kappa$ and therefore $\kappa \geq (\nu^+)^M = j_U(\nu^+) = j_U(\kappa) \geq \kappa$.
 \end{proof}

 Define ${\blacktriangleleft} = j_U(\lhd)$.

 \begin{claim*}
  ${\blacktriangleleft}={\lhd\cap M}$ and $\POT{\kappa}^M$ is $\lhd$-downwards closed in $\POT{\kappa}$. 
 \end{claim*}

 \begin{proof}
 By the above claim and elementarity, we know that $$I({\blacktriangleleft})^M ~ = ~ \Set{A\in M}{\varphi(A,\kappa)^M}$$ and therefore $\Sigma_1$-upwards absoluteness implies that  $I({\blacktriangleleft})^M\subseteq I(\lhd)$. 
 
 Pick $x,y\in\POT{\kappa}^M$. First, assume that $x \blacktriangleleft y$. Then there is $A\in I({\blacktriangleleft})^M$ with $x\in A$ and $y\notin A$. By the above remarks, we have $A\in I(\lhd)$ and therefore $x\lhd y$. In the other direction, assume that $x\lhd y$ holds. Set $A=\Set{a\in\POT{\kappa}^M}{a \blacktriangleleft y}$ and $B=A\cup\{y\}$. Then $A,B\in I({\blacktriangleleft})^M\subseteq I(\lhd)$ and therefore $A=\Set{a\in\POT{\kappa}}{a\lhd y}$. In particular, $x\in A$ implies that $x\lhd y$. Moreover, this argument also shows that $\POT{\kappa}^M$ is $\lhd$-downwards closed in $\POT{\kappa}$. 
 \end{proof}

 Since $\nu>\delta$ is $\delta$-inaccessible, we have $2^\delta<\nu\leq\kappa$ and therefore there is a subset $x$ of $\kappa$ that is not contained in $M$. Set $\beta=\rank{x}{\lhd}$.  By elementarity, we have $$\otp{\POT{\kappa}^M,\blacktriangleleft} ~ > ~ j_U(\beta) ~ \geq ~ \beta$$ and we can find $y\in\POT{\kappa}^M$ with $\rank{y}{\blacktriangleleft}=\beta$. Since the above claim shows that ${\blacktriangleleft}\subseteq {\lhd}$, we have $\rank{y}{\lhd}\geq\beta$ and hence $x\neq y$ implies that $x\lhd y$. But the above claim also shows that $\POT{\kappa}^M$ is $\lhd$-downwards closed in $\POT{\kappa}$ and we can conclude that $x\in M$, a contradiction.  
\end{proof}

In models of the $\GCH$, the statement of the last lemma greatly simplifies.

\begin{corollary}\label{corollary:SimplificationGCHNonMeasurable}
 Assume that $\GCH$ holds above a measurable cardinal $\delta$ and $\kappa\geq\delta$ is a cardinal with the property that there is a good $\Sigma_1(\kappa)$-definable wellorder of $\POT{\kappa}$. 
 \begin{enumerate}
  \item There is a cardinal $\nu$ with $\cof{\nu}=\delta$ and $\nu\leq\kappa\leq\nu^+$. 

  \item There are no measurable cardinals in the interval $(\delta,\kappa]$. 
 \end{enumerate}
\end{corollary}

\begin{proof}
 (i) We may assume that $\kappa>\delta^+$. If $\kappa$ is a limit cardinal, then our $\GCH$-assumption implies that $\kappa$ is $\delta$-inaccessible and therefore Lemma \ref{lemma:MeasurableWOabove} shows that $\cof{\kappa}=\delta$. Now, assume that $\kappa=\nu^+$ with $\nu>\delta$. Then $\nu$ is singular, because otherwise $\kappa$ would be $\delta$-inaccessible and Lemma \ref{lemma:MeasurableWOabove} would imply that there are no good $\Sigma_1(\kappa)$-wellorderings of $\POT{\kappa}$. Since the $\GCH$ holds above $\delta$, this implies that $\nu$ is $\delta$-inaccessible and Lemma  \ref{lemma:MeasurableWOabove} implies that $\cof{\nu}=\delta$. 

 (ii) Assume, towards a contradiction, that there is a measurable cardinal $\varepsilon$ in the interval $(\delta,\kappa]$. By applying the first part of the corollary with both $\delta$ and $\varepsilon$, we find cardinals $\mu$ and $\nu$ with $\cof{\mu}=\delta$, $\cof{\nu}=\varepsilon$, $\mu\leq\kappa\leq\mu^+$ and $\nu\leq\kappa\leq\nu^+$. But then $\mu=\nu$ and hence  $\delta=\varepsilon$, a contradiction. 
\end{proof}

In order to show that the above lemmas can consistently state all limitations that a measurable cardinal imposes on the existence of good $\Sigma_1$-wellorders, we will use classical results of Kunen and Silver to prove the following result in Section \ref{section:SingleMeasurable}.

\begin{theorem}\label{theorem:Main}
 Assume that $\delta$ is a measurable cardinal and $U$ is a normal ultrafilter on $\delta$ such that $\VV=\LL[U]$ holds. Given an infinite cardinal $\kappa$, there is a good $\Sigma_1(\kappa)$-wellordering of $\POT{\kappa}$ if and only if either $\omega<\kappa<\delta$ or $\nu\leq\kappa\leq\nu^+$ for some cardinal $\nu$ with $\cof{\nu}=\delta$. 
\end{theorem}

The above result directly shows that the desired characterization of the class of all cardinals $\kappa$ with the property that there is a good $\Sigma_1(\kappa)$-wellordering of $\POT{\kappa}$ holds true in models of the form $\LL[U]$.

\begin{corollary}\label{corollary:CharacterizeAllGoodWOinLU}
 In the setting of Theorem \ref{theorem:Main}, the following statements are equivalent for every infinite cardinal $\kappa$: 
 \begin{enumerate}
    \item Either $\kappa$ is countable or there is a $\delta$-inaccessible cardinal $\nu$ with $\cof{\nu}\neq\delta$ and $\nu\leq\kappa\leq\nu^+$. 

  \item There is no good $\Sigma_1(\kappa)$-wellordering of $\POT{\kappa}$. 
 \end{enumerate}
\end{corollary}

\begin{proof}
The implication from (i) to (ii) is proved in Lemma \ref{lemma:MeasurableWOReals} and Lemma \ref{lemma:MeasurableWOabove}. 
We now assume that (i) fails and (ii) holds. Then Theorem \ref{theorem:Main} implies that $\kappa>\delta^+$. Since the $\GCH$ holds in $\LL[U]$ (see {\cite[Theorem 20.3]{MR1994835}}) and every limit cardinal greater than $\delta$ is $\delta$-inaccessible, we know that $\kappa$ is not a limit cardinal, because otherwise Theorem \ref{theorem:Main} would imply that  $\cof{\kappa}\neq\delta$ and (i) would hold. Hence $\kappa=\nu^+$ with $\nu>\delta$ and $\nu$ is singular, because the $\GCH$ implies that successors of regular cardinals above $\delta$ are $\delta$-inaccessible. In this situation, Theorem \ref{theorem:Main} implies that $\cof{\nu}\neq\delta$ and hence (i) holds, a contradiction. 
\end{proof}

Motivated by the above results, we will prove the following result in Section \ref{section:Core}. It  uses the \emph{Dodd-Jensen core model $\KK^{DJ}$} (see \cite{MR611394}) to show that measurability can be considered the smallest large cardinal property that implies the non-existence of good $\Sigma_1$-wellorders at certain uncountable cardinals.

\begin{lemma}\label{lemma:WOinKDJ}
 Assume that $\VV=\KK^{DJ}$ holds. If $\kappa$ is an uncountable cardinal, then there is a good $\Sigma_1(\kappa)$-wellordering of $\POT{\kappa}$.  
\end{lemma}

Finally, we consider the influence of the existence of two measurable cardinals on the existence of good $\Sigma_1$-wellorders at uncountable cardinals. Section \ref{section:TwoMeasurables} contains the proof of the following result that directly generalizes Theorem \ref{theorem:Main} to this setting.

\begin{theorem}\label{theorem:TwoMeasurables}
 Let $\delta_0<\delta_1$ be measurable cardinals, let $U_0$ be a normal ultrafilter on $\delta_0$ and let $U_1$ be a normal ultrafilter on $\delta_1$. Assume that $\VV=\LL[U_0,U_1]$ holds. Given an infinite cardinal $\kappa$, there is a good $\Sigma_1(\kappa)$-wellordering of $\POT{\kappa}$ if and only if either $\omega<\kappa<\delta_0$ or $\nu\leq\kappa\leq\nu^+$ for some cardinal $\nu<\delta_1$ with $\cof{\nu}=\delta_0$.  
\end{theorem}

This theorem shows that, parallel to the above results, it is possible that there are two measurable cardinals and good $\Sigma_1$-wellorderings exists at all cardinals that are not ruled out by the above lemmas.

\begin{corollary}
 In the setting of Theorem \ref{theorem:TwoMeasurables}, the following statements are equivalent for every infinite cardinal $\kappa$: 
 \begin{enumerate}
    \item Either $\kappa$ is countable or there is an $i<2$ and a $\delta_i$-inaccessible cardinal $\nu$ with the property that $\cof{\nu}\neq\delta_i$ and $\nu\leq\kappa\leq\nu^+$. 

  \item There is no good $\Sigma_1(\kappa)$-wellordering of $\POT{\kappa}$. 
 \end{enumerate}
\end{corollary}

\begin{proof}
The implication from (i) to (ii) is proved in Lemma \ref{lemma:MeasurableWOReals} and Lemma \ref{lemma:MeasurableWOabove}. 
 Assume that (i) fails and (ii) holds. Then Theorem \ref{theorem:TwoMeasurables} implies that $\kappa>\delta_0^+$.

 First, assume that $\kappa$ is a limit cardinal. Then $\cof{\kappa}=\delta_0$, because otherwise the $\GCH$ in $\LL[U_0,U_1]$ (see {\cite[Theorem 3.6]{MR0344123}}) would imply that $\kappa$ is a $\delta_0$-inaccessible cardinal of cofinality different from $\delta_0$, contradicting the failure of (i). In this situation, Theorem \ref{theorem:TwoMeasurables} implies that $\kappa>\delta_1$ and the $\GCH$ implies that $\kappa$ is $\delta_1$-inaccessible, contradicting our assumptions.

 These computations show that $\kappa=\nu^+$ with $\nu>\delta_0$. Then $\nu$ is not regular, because otherwise the $\GCH$ would imply that $\kappa$ is $\delta_0$-inaccessible. Since $\nu$ is  a limit cardinal greater than $\delta_0$,  we know that $\nu$ is $\delta_0$-inaccessible and therefore  our assumption implies that $\cof{\nu}=\delta_0$. Then Theorem \ref{theorem:TwoMeasurables} shows that $\nu>\delta_1$ and hence $\nu$ is a $\delta_1$-inaccessible cardinal with $\cof{\nu}\neq\delta_1$, contradicting our assumption that (i) fails.  
\end{proof}

We outline the structure of this paper. In Section \ref{section:Core}, we briefly introduce some terminology to talk about the structure of so-called \emph{short core models} and then show that the canonical wellorders of these models are good $\Sigma_1$-wellorders. Section \ref{section:SingleMeasurable} contains the proof of Theorem \ref{theorem:Main} that relies on classical results of Kunen and Silver on the structure of models of the form $\LL[U]$. In Section \ref{section:TwoMeasurables}, we use results of Koepke on the structure of short core models to generalize the previous results to models of the form $\LL[U_0,U_1]$ and prove Theorem \ref{theorem:TwoMeasurables}. We close this paper by listing some open questions motivated by these results in Section \ref{section:questions}.


\section{Models of the form $\KK[D]$}\label{section:Core}

In this section, we construct good $\Sigma_1$-wellorderings in certain canonical models of set theory called \emph{short core models}. These models were studied by Koepke in \cite{PeterThesis} and \cite{MR926749}. Our results will rely on the outline of the structure theory of these models presented in \cite{MR926749}. Even though we will only consider short core model whose measure sequence has length at most two, we introduce the general terminology needed to construct and study these model. This will allow us to directly refer to the results presented in \cite{MR926749}.

\begin{definition}
 Let $D$ be a class. 
 \begin{enumerate}
  \item $D$ is \emph{simple} if the following statements hold: 
 \begin{enumerate}
  \item If $x\in D$, then there is $\delta\in\On$ with $x=\langle\delta,a\rangle$ with $a\subseteq\delta$.
  
  \item If $\langle\delta,a\rangle\in D$, then $\langle\delta,\delta\rangle\in D$.  
 \end{enumerate}
 In the above situation, we define $\dom{D}=\Set{\delta\in\On}{\langle\delta,\delta\rangle\in D}$ and $D(\delta)=\Set{a\subseteq\delta}{\langle\delta,a\rangle\in D}$ for all $\delta\in\dom{D}$. 
 
  \item $D$ is a \emph{sequence of measures} if $D$ is simple and $D(\delta)$ is a normal ultrafilter on $\delta$ for every $\delta\in\dom{D}$. 
 \end{enumerate}
\end{definition}

\begin{definition}
 Let $D$ be a simple set.

 \begin{enumerate}
  \item We say that $M=\langle\betrag{M},F_M\rangle$ is a \emph{premouse over $D$} if   
  the following statements hold: 
  \begin{enumerate}
   \item $\betrag{M}$ is a transitive set and $F_M$ is a simple set with the property that $\sup{(\dom{D})}<\min{(\dom{F_M})}\in\betrag{M}$. 

   \item $\langle\betrag{M},\in,F_M\rangle\models\anf{\textit{$F_M$ is a sequence of measures}}$. 

   \item $\betrag{M}=\JJ_{\alpha(M)}[D,F_M]$ for some ordinal $\alpha(M)$. 
  \end{enumerate}
  In the above situation, we define $\meas{M}=\dom{F_M}\cap{(\omega\cdot\alpha(M))}$ and $\lp{M}=\HHH{\min{(\meas{M}})}^{\betrag{M}}$.

 \item Given a premouse $M$ over $D$ and $\delta\in\meas{M}$, a premouse $M_*$ over $D$ is the \emph{ultrapower of $M$ at $\delta$} if there is a unique map $\map{j}{\betrag{M}}{\betrag{N}}$ with the following properties, and $M_*=N$ holds for this $N$.  
  \begin{enumerate}
   \item $\map{j}{\langle\betrag{M},\in,D,F_M\rangle}{\langle\betrag{N},\in,D,F_N\rangle}$ is $\Sigma_1$-elementary. 
   
   \item $\betrag{N}=\Set{j(f)(\delta)}{f\in{}^\delta\betrag{M}\cap\betrag{M}}$. 
   
   \item $F_M(\delta)\cap\betrag{M}=\Set{x\in\POT{\delta}\cap\betrag{M}}{\delta\in j(x)}$. 
  \end{enumerate}
  If such a premouse exists, then we denote it by $\Ult{M}{F(\delta)}$ and we call the corresponding map $j$ the \emph{ultrapower embedding of $M$ at $\delta$}.

 \item Given a premouse $M$ over $D$ and a function $\map{I}{\lambda}{\On}$ with $\lambda\in\On$, a system $$\mathrm{It}(M,I) ~ = ~ \langle\seq{M_\alpha}{\alpha\leq\lambda}, ~ \seq{j_{\alpha,\beta}}{\alpha\leq\beta\leq\lambda}\rangle$$ is called the \emph{iterated ultrapower of $M$ by $I$} if the following statements hold for all $\gamma\leq\lambda$: 
  \begin{enumerate}
   \item $M=M_0$ and $M_\gamma$ is a premouse over $D$. 
   
   \item Given $\alpha\leq\beta\leq\gamma$, $\map{j_{\beta,\gamma}}{\betrag{M_\beta}}{\betrag{M_\gamma}}$ is a function, $j_{\gamma,\gamma}=\id_{\betrag{M_\gamma}}$ and $j_{\alpha,\gamma}=j_{\beta,\gamma}\circ j_{\alpha,\beta}$. 
   
   \item If $\gamma<\lambda$ and $I(\gamma)\in\meas{M_\gamma}$, then $M_{\gamma+1}=\Ult{M_\gamma}{F_{M_\gamma}(I(\gamma))}$ and $j_{\gamma,\gamma+1}$ is the ultrapower embedding of $M_\gamma$ at $F_{M_\gamma}(I(\gamma))$. In the other case, if  $\gamma<\lambda$ and $I(\gamma)\notin\meas{M_\gamma}$, then $M_\gamma=M_{\gamma+1}$ and $j_{\gamma,\gamma+1}=\id_{\betrag{M_\gamma}}$. 
   
   \item If $\gamma\in\Lim$, then $\langle\langle M_\gamma,\in,D,F_{M_\gamma}\rangle, ~ \seq{j_{\beta,\gamma}}{\beta<\gamma}\rangle$ is a direct limit of the directed system $\langle\seq{M_\beta,\in,D,F_{M_\beta}}{\beta<\gamma}, ~ \seq{j_{\alpha,\beta}}{\alpha\leq\beta<\gamma}\rangle$. 
  \end{enumerate}
  In this situation, we let $M_I$ denote $M_\lambda$ and let $j_I$ denote $j_{0,\lambda}$. 
  
  \item A premouse $M$ over $D$ is \emph{iterable} if the system $\mathrm{It}(M,I)$ exists for every function $\map{I}{\lambda}{\On}$ with $\lambda\in\On$. 
  
  \item A premouse $M$ over $D$ is \emph{short} if one of the following statements holds:
   \begin{enumerate}
    \item $D=\emptyset$ and $\mathrm{otp}(\meas{M}\cap\gamma)<\min(\meas{M})$ for all $\gamma\in\betrag{M}\cap\On$. 
    
    \item $D\neq\emptyset$ and $\mathrm{otp}(\meas{M})\leq\min(\dom{D})$. 
   \end{enumerate}
   
  \item A $D$-mouse is an iterable short premouse over $D$. 
  
  \item If either $D=\emptyset$ or $\otp{\dom{D}}\leq\min{(\dom{D})}$, then we define $$K[D] ~ = ~ \Set{\lp{M}}{\textit{$M$ is a $D$-mouse}}$$ and, given $x,y\in\KK[D]$, we write $x<_{\KK[D]} y$ to denote that $x<_{\LL[D,F_M]}y$ holds for every $D$-mouse $M$ with $x,y\in\lp{M}$. 
 \end{enumerate}
\end{definition}

In the following, we omit the parameter $D$ from the above notations if it is equal to the empty set, i.e. \emph{mouse} means $\emptyset$-mouse, we write $\KK$ instead of $\KK[\emptyset]$ etc. 

\begin{lemma}\label{lemma:GoodWOinK}
 Assume that $D$ is a simple set with the property that either $D=\emptyset$ or $\otp{\dom{D}}\leq\min{(\dom{D})}$. If $\VV=\KK[D]$ holds and $\kappa$ is an uncountable cardinal, then there is a good $\Sigma_1(\kappa,D)$-wellordering of $\POT{\kappa}$. 
\end{lemma}

\begin{proof}
 By {\cite[Theorem 3.4]{MR926749}}, we know that $<_{\KK[D]}$ is a wellordering of $\VV$. 
 Let $\lhd$  denote the restriction of $<_{\KK[D]}$ to $\POT{\kappa}$. 

 \begin{claim*}
   If $M$ and $N$ are $D$-mice and $x_0,x_1\in\lp{M}\cap\lp{N}$, then  $x_0<_{\LL[D,F_M]}x_1$ implies that $x_0<_{\LL[D,F_N]}x_1$.
 \end{claim*}
 
 \begin{proof}[Proof of the Claim]
 We have that $x$ and $y$ are comparable in $<_{\KK[D]}$. Assuming that $x_0<_{\LL[D,F_M]}x_1$, it follows that $x_0<_{\KK[D]}x_1$. Thus also $x_0<_{\LL[D,F_N]}x_1$  by the definition of  $<_{\KK[D]}$.   
 \end{proof}

 The above claim shows that $x\lhd y$ holds if and only if there is a $D$-mouse $M$ with $x,y\in\lp{M}$ and $x<_{\LL[D,F_M]}y$. Moreover, every $D$-mouse $M$ with $\kappa\in\lp{M}$ is $\lhd$-downwards closed. In particular, a set $X$ is contained in the set $I(\lhd)$ of all initial segments of $\lhd$ if and only if there is a $D$-mouse $M$ such that $\kappa,X\in\lp{M}$ and $X$ is an initial segment of the restriction of $<_{\LL[D,F_M]}$ to $\POT{\kappa}$ in $\langle\betrag{M},\in,D,F_M\rangle$.

 Since {\cite[Theorem 2.7]{MR926749}} shows that there is a finite fragment $\sf F$ of $\ZFC$ such that  the statement \anf{\emph{$M$ is a $D$-mouse}} is absolute between $\VV$ and transitive models of $\sf F$ containing $\kappa$ and $M$, we can conclude that it is possible to define the collection of all $D$-mice by a $\Sigma_1$-formula with parameters $\kappa$ and $D$. By the above computations, this shows  that the set $I(\lhd)$ can be defined by a $\Sigma_1$-formula that uses only $\kappa$ and $D$ as parameters. 
\end{proof}

\begin{proof}[Proof of Lemma \ref{lemma:WOinKDJ}]
 Assume that $\VV=\KK^{DJ}$ holds and let $\kappa$ be an uncountable cardinal. If there are no mice, then the results of {\cite[Section 6]{MR611394}} show that $\KK^{DJ}=\LL$ and the restriction of the canonical wellordering of $\LL$ to $\POT{\kappa}$ is a good $\Sigma_1(\kappa)$-wellordering. 
Hence we may assume that there is a mouse. In this situation, results of Dodd and Jensen (see  {\cite[p. 238]{MR730856}}) show that $\KK^{DJ}$ is equal to the union of all $\lp{M}$, where $M$ is a mouse with $\otp{\meas{M}}=1$. In particular, $\KK^{DJ}=\KK$ holds in this case. In this situation, Lemma \ref{lemma:GoodWOinK} directly implies that there is a good $\Sigma_1(\kappa)$-wellordering of $\POT{\kappa}$. 
\end{proof}

The above result allows us to construct good $\Sigma_1$-wellorderings in canonical inner models at uncountable cardinals that are less than or equal to the unique measurable cardinal in these models.

\begin{corollary}\label{corollary:GoodWObelowMeasurable}
  Assume that $\delta$ is a measurable cardinal and $U$ is a normal ultrafilter on $\delta$ such that $\VV=\LL[U]$ holds. If $\kappa\leq\delta$ is an uncountable cardinal, then there is a good $\Sigma_1(\kappa)$-definable wellorder of $\POT{\kappa}$.  
\end{corollary} 

\begin{proof} 
 By {\cite[Section 6]{MR611394}}, we know that our assumptions imply that $\KK^{DJ}$ is equal to the intersection of all iterated ultrapowers of $\langle\VV,\in,U\rangle$. In particular, we know that  $\POT{\delta}\subseteq\KK^{DJ}$. In this situation, Lemma \ref{lemma:WOinKDJ} directly implies that there is a good $\Sigma_1(\kappa)$-definable wellorder of $\POT{\kappa}$ for every uncountable cardinal $\kappa\leq\delta$.  
\end{proof}

\begin{corollary}\label{corollary:LowWOTwoMeasurables}
 Let $\delta_0<\delta_1$ be measurable cardinals, let $U_0$ be a normal ultrafilter on $\delta_0$ and let $\delta_1$ be a normal ultrafilter on $\delta_1$. Assume that $\VV=\LL[U_0,U_1]$ holds. If $\kappa\leq\delta_0$ is an uncountable cardinal, then there is a good $\Sigma_1(\kappa)$-definable wellorder of $\POT{\kappa}$.  
\end{corollary} 

\begin{proof}
 Since $\POT{\delta_0}\subseteq\Ult{\VV}{U_0}$, we know that for every subset $x$ of $\delta_0$, there is a mouse $M$ with $x\in\lp{M}$ and $\otp{\meas{M}}=2$. Hence $\POT{\delta_0}\subseteq\KK$ and Lemma \ref{lemma:GoodWOinK} shows that there is a good $\Sigma_1(\kappa)$-wellordering of $\POT{\kappa}$ for every uncountable cardinal $\kappa\leq\delta_0$.
\end{proof}


\section{Models of the form $L[U]$}\label{section:SingleMeasurable}

In this section, we will prove Theorem \ref{theorem:Main}. Throughout this section, we will work in the setting of the theorem: \emph{$\delta$ is a measurable cardinal and $U$ is a normal ultrafilter on $\delta$ with the property that $\VV=\LL[U]$ holds}.

We start by showing that the unique normal filter in $\LL[U]$ is $\Sigma_1$-definable from certain cardinals above the unique measurable cardinal of  $\LL[U]$. These arguments heavily rely on results of Kunen and Silver (see \cite{MR0277346} and \cite{MR0278937}) on the structure of models of the form $\LL[U]$. In the following, we will refer to the presentation of these results in {\cite[Section 20]{MR1994835}}.

\begin{lemma}\label{lemma:GoodWOSingular}
  In the setting of Theorem \ref{theorem:Main}, if $\nu>\delta$ is a cardinal with $\cof{\nu}=\delta$, then the set $\{U\}$ is $\Sigma_1(\nu)$-definable. 
\end{lemma}

\begin{proof}
  Pick a $\Sigma_1$-formula $\Phi(v_0,\ldots,v_3)$ with the property that $\Phi(M,\varepsilon,F,\nu)$ is equivalent to the conjunction of the following statements: 
 \begin{enumerate}
  \item $M$ is a transitive model of $\ZFC^-$ with $\nu\in M$. 

  \item $\varepsilon<\nu$ is a measurable cardinal in $M$ with $\cof{\nu}^M=\varepsilon$. 

  \item $F\in M$ is a normal ultrafilter on $\varepsilon$ in $M$ and $M=\LL_\alpha[F]$ for some $\alpha\in\On$. 
 \end{enumerate}
 Our assumptions imply that $\Phi(L_{\nu^{++}}[U],\delta,U,\nu)$ holds.

  \begin{claim*}
  If $\Phi(M,\varepsilon,F,\nu)$ holds, then $\cof{\nu}^{\LL[F]}=\varepsilon$, $\varepsilon$ is a measurable cardinal in $L[F]$ and $F$ is a normal ultrafilter on $\varepsilon$ in $\LL[F]$.
 \end{claim*}
 
  \begin{proof}[Proof of the Claim]
  Since $\nu$ is a strong limit cardinal greater than $\varepsilon$ in $\VV$, the condensation principle for $\LL[F]$ implies that $\HHH{\nu}^{\LL[F]}\subseteq\LL_\nu[F]\subseteq M$. This shows that $\varepsilon$ is a measurable cardinal in $\LL[F]$ and $F$ is a normal ultrafilter on $\varepsilon$ in $\LL[F]$. Moreover, since $\varepsilon$ is regular in $\LL[F]$, we get $\varepsilon=\cof{\nu}^{\LL[F]}$, because otherwise $\cof{\nu}^{\LL[F]}<\varepsilon$ would imply that $\cof{\varepsilon}^{\LL[F]}<\varepsilon=\cof{\varepsilon}^M$ holds and this would imply that $\POT{\varepsilon}^M\neq\POT{\varepsilon}^{\LL[F]}$. 
 \end{proof}

  \begin{claim*}
  If $\Phi(M,\varepsilon,F,\nu)$ holds, then $\delta=\varepsilon$ and $F=U$. 
 \end{claim*}

 \begin{proof}[Proof of the Claim]
  Assume, towards a contradiction, that $\delta\neq\varepsilon$ holds. Since we have  $\delta=\cof{\nu}\leq\cof{\nu}^{\LL[F]}=\varepsilon$, we can conclude that $\delta<\varepsilon$. By our assumptions, $\langle\LL[U],\in,U\rangle$ is \emph{the $\delta$-model} (in the sense of {\cite[Section 20]{MR1994835}}) and the above claim shows that $\langle\LL[F],\in,F\rangle$ is the $\varepsilon$-model. 
    Then {\cite[Theorem 20.12]{MR1994835}} allows us to find a $0<\tau\leq\varepsilon$ with the property that, if $$\langle\seq{\langle N_\alpha,\in,F_\alpha\rangle}{\alpha\leq\tau},\seq{\map{j_{\alpha,\beta}}{N_\alpha}{N_\beta}}{\alpha\leq\beta\leq\tau}\rangle$$ denotes the corresponding system of iterated ultrapowers of $\langle\LL[U],\in,U\rangle$, then we have $N_\tau=\LL[F]$, $j_{0,\tau}(\delta)=\varepsilon$ and $F=F_\tau$. In this situation, we can use {\cite[Corollary 19.7]{MR1994835}} to see that $j_{0,\tau}(\mu)<\nu$ holds for every $\mu<\nu$. Since $\crit{j_{0,\tau}}=\delta$, this allows us to define a continuous, cofinal and strictly increasing map $\map{c}{\delta}{\nu}$ with the property that $j_{0,\tau}(c(\gamma))=c(\gamma)$ holds for every $\gamma<\delta$. But then we have $c=j_{0,\tau}(c)\restriction\delta\in\LL[F]$ and therefore $\cof{\nu}^{\LL[F]}\leq\delta<\varepsilon$, contradicting the above claim. 
  
  The above computations show that $\delta=\varepsilon$. Using {\cite[Theorem 20.10]{MR1994835}} and the above claim, we can conclude that we also have $F=U$. 
\end{proof}

The last claim shows that $U$ is the unique set $F$ such that there are $M$ and $\varepsilon$ with the property that $\Phi(M,\varepsilon,F,\nu)$ holds. This shows that the set $\{U\}$ is definable by a $\Sigma_1$-formula with parameter $\nu$. 
\end{proof}

In the following lemma, we prove the analog of the above result for successors of singular cardinals of cofinality equal to the unique measurable cardinal in $\LL[U]$.

\begin{lemma}\label{lemma:GoodWOSuccSingular}
  In the setting of Theorem \ref{theorem:Main}, if $\mu\geq\delta$ is a cardinal with $\cof{\mu}=\delta$, then the set $\{U\}$ is $\Sigma_1(\mu^+)$-definable. 
\end{lemma}

\begin{proof}
   Set $\kappa=\mu^+$ and let $$\langle\seq{\langle N_\gamma,\in,F_\gamma\rangle}{\gamma<\kappa},\seq{\map{j_{\beta,\gamma}}{N_\beta}{N_\gamma}}{\beta\leq\gamma<\kappa}\rangle$$ denote the corresponding system of iterated ultrapowers of $\langle\VV,\in,U\rangle$. We define $\delta_\gamma=j_{0,\gamma}(\delta)$ and $\mu_\gamma=j_{1,\gamma}(\mu)$ for all $0<\gamma<\kappa$. Since $\mu$ is a strong limit cardinal greater than $\delta$, we have $\delta_1<\mu$ and this implies that $\delta_\gamma<\mu_\gamma$ for all $0<\gamma<\kappa$.

 \begin{claim*}
  If $0<\gamma<\kappa$, then $\cof{\mu_\gamma}^{N_\gamma}=\delta$ and $\kappa=(\mu_\gamma^+)^{N_\gamma}=j_{1,\gamma}(\kappa)$. 
 \end{claim*}

 \begin{proof}[Proof of the Claim]
  We prove the claim by induction on $0<\gamma<\kappa$. 

  Since $N_1=\Ult{\VV}{U}$, we have ${}^\delta N_1\subseteq N_1$, $\cof{\mu}^{N_1}=\delta$ and $$\kappa ~ = ~ \mu^+ ~ = ~ \mu^\delta ~ \leq ~ (\mu^\delta)^{N_1} ~ = ~ (\mu^+)^{N_1} ~ \leq ~ \mu^+ ~ = ~ \kappa.$$

  Next, assume that $\gamma=\bar{\gamma}+1$. Then $N_\gamma=\Ult{N_{\bar{\gamma}}}{F_{\bar{\gamma}}}$ and $\map{j_{\bar{\gamma},\gamma}}{N_{\bar{\gamma}}}{N_\gamma}$ is the corresponding ultrapower embedding in $N_{\bar{\gamma}}$.  Our induction hypothesis implies that $\mu_{\bar{\gamma}}>\delta_{\bar{\gamma}}$ is a strong limit cardinal of cofinality $\delta$ in $N_{\bar{\gamma}}$. Then $j_{\bar{\gamma},\gamma}(\alpha)<\mu_{\bar{\gamma}}$ for all $\alpha<\mu_{\bar{\gamma}}$ and there is a cofinal function $\map{c}{\delta}{\mu_{\bar{\gamma}}}$ in $N_{\bar{\gamma}}$ with the property that  $j_{\bar{\gamma},\gamma}(c(\beta))=c(\beta)$ holds for all $\beta<\delta$. Since $\delta<\delta_{\bar{\gamma}}=\crit{j_{\bar{\gamma},\gamma}}$, we have $j_{\bar{\gamma},\gamma}(c)=c$ and this implies that $$\mu_\gamma ~ = ~ j_{1,\gamma}(\mu) ~ = ~ j_{\bar{\gamma},\gamma}(\mu_{\bar{\gamma}}) ~ = ~ \mu_{\bar{\gamma}} ~ < ~ \kappa$$ and $\cof{\mu_\gamma}^{N_\gamma}=j_{1,\gamma}(\delta)=\delta$. Finally, we have $$\kappa ~ \geq ~ (\mu_\gamma^+)^{N_\gamma} ~ = ~ j_{\bar{\gamma},\gamma}\left((\mu^+_{\bar{\gamma}})^{N_{\bar{\gamma}}}\right) ~ = ~ j_{\bar{\gamma},\gamma}(\kappa) ~ \geq ~ \kappa.$$

 Now, assume that $\gamma\in\Lim\cap\kappa$. Since $\delta<\delta_1=\crit{j_{1,\gamma}}$,  elementarity and our induction hypothesis imply that $$\cof{\mu_\gamma}^{N_\gamma} ~ = ~ j_{1,\gamma}\left(\cof{\mu}^{N_1}\right) ~ = ~ j_{1,\gamma}(\delta) ~ = ~ \delta.$$ If $\beta<\mu_\gamma=j_{1,\gamma}(\mu)$, then we can find $\gamma_\beta<\gamma$ and $\alpha_\beta<\mu_{\gamma_\beta}$ with $\beta=j_{\gamma_\beta,\gamma}(\alpha_\beta)$. This yields an injection of $\mu_\gamma$ into $\gamma\cdot\sup_{\bar{\gamma}<\gamma}\mu_{\bar{\gamma}}$. Since our induction hypothesis implies that this product of ordinals is smaller than $\kappa$ and $\kappa=(\mu^+)^{N_1}$, we can conclude that $\mu_\gamma<\kappa$ and  this implies that $$\kappa ~ \geq ~ (\mu_\gamma^+)^{N_\gamma} ~ = ~ j_{1,\gamma}\left((\mu^+)^{N_1}\right) ~ = ~ j_{1,\gamma}(\kappa) ~ \geq ~ \kappa $$  holds. 
 \end{proof}

  Pick a $\Sigma_1$-formula $\Psi(v_0,\ldots,v_4)$ with the property that $\Psi(M,\varepsilon,F,\nu,\kappa)$ is equivalent to the conjunction of the following statements: 
 \begin{enumerate}
  \item $M$ is a transitive model of $\ZFC^-$ with $\kappa\in M$. 

  \item $\varepsilon<\kappa$ is a measurable cardinal in $M$. 

  \item $F\in M$ is a normal ultrafilter on $\varepsilon$ in $M$ and $M=\LL_\alpha[F]$ for some $\alpha\in\On$. 
  
  \item $\nu$ is a cardinal in $M$ with $\varepsilon<\nu<\kappa$, $\varepsilon=\cof{\nu}^M$ and $\kappa=(\nu^+)^M$. 
 \end{enumerate}
 Note that our assumptions imply that $\Psi(L_{\kappa^{++}}[U],\delta,U,\mu,\kappa)$ holds.

   \begin{claim*}
  If $\Psi(M,\varepsilon,F,\nu,\kappa)$ holds, then $\nu$ is a cardinal in $\LL[F]$, $\kappa=(\nu^+)^{\LL[F]}$, $\cof{\nu}^{\LL[F]}=\varepsilon$, $\varepsilon$ is an inaccessible cardinal in $L[F]$ and $F$ is a normal ultrafilter on $\varepsilon$ in $\LL[F]$.
 \end{claim*}
 
 \begin{proof}[Proof of the Claim]
  Since $\kappa$ is a regular cardinal greater than $\nu$ in $\VV$, we can use the condensation principle for $\LL[F]$ to show that $\POT{\nu}^{\LL[F]}\subseteq\LL_\kappa[F]\subseteq M$. Together with our assumptions, this observation implies the above conclusions. 
 \end{proof}

 \begin{claim*}
  If $\Psi(M,\varepsilon,F,\nu,\kappa)$ holds, then $\delta=\varepsilon$ and $F=U$.  
 \end{claim*}
 
 \begin{proof}[Proof of the Claim]
  First, assume that $\delta\neq\varepsilon$. Since the above claim shows that $\langle\LL[F],\in,F\rangle$ is the $\varepsilon$-model, {\cite[Theorem 20.12]{MR1994835}} shows that either $\langle\LL[F],\in,F\rangle$ is an iterate of $\langle\LL[U],\in,U\rangle$ or $\langle\LL[U],\in,U\rangle$ is an iterate of $\langle\LL[F],\in,F\rangle$. But we also know that  $F$ is an element of $\LL[U]$ and this implies that the first option holds.  Hence our assumption yields $0<\gamma\leq\varepsilon<\kappa$ with $\LL[F]=N_\gamma$, $\varepsilon=\delta_\gamma$ and $F=F_\gamma$. In this situation, the above claims imply that $\mu_\gamma$ is a cardinal of cofinality $\delta$ in $\LL[F]$, $\nu$ is a cardinal of cofinality $\varepsilon$ in $\LL[F]$ and  $(\nu^+)^{\LL[F]}=\kappa=(\mu_\gamma^+)^{\LL[F]}$ holds. But this implies that $\mu_\gamma=\nu$ and hence $\delta=\varepsilon$, a contradiction. 
 
   Since these computations show that $\delta=\varepsilon$, a combination of the last claim and {\cite[Theorem 20.10]{MR1994835}} allows us to conclude that $F=U$. 
 \end{proof}

 Using the above claim, we see that $U$ is the unique set $F$ with the property that $\Psi(M,\varepsilon,F,\nu,\kappa)$ holds for some $M$, $\varepsilon$ and $\nu$. In particular, it is possible to define the set $\{U\}$ using a $\Sigma_1$-formula with parameter $\kappa$.  
\end{proof}

\begin{proof}[Proof of Theorem \ref{theorem:Main}]
 Assume that $\delta$ is a measurable cardinal and $U$ is a normal ultrafilter on $\delta$ such that $\VV=\LL[U]$ holds and $\kappa$ is an infinite cardinal.

 First, assume that there is a good $\Sigma_1(\kappa)$-wellordering of $\POT{\kappa}$. Then Lemma \ref{lemma:MeasurableWOReals} implies that $\kappa$ is uncountable. Since the $\GCH$ holds in $\LL[U]$, we can use Corollary \ref{corollary:SimplificationGCHNonMeasurable} to conclude that either $\omega<\kappa<\delta$ or there is a cardinal $\nu$ with $\cof{\nu}=\delta$ and $\nu\leq\kappa\leq\nu^+$.

 In the other direction, if $\omega<\kappa<\delta$, then Corollary  \ref{corollary:GoodWObelowMeasurable} directly implies that there is a good $\Sigma_1(\kappa)$-wellordering of $\POT{\kappa}$.   Finally, assume that there is a cardinal $\nu$ with $\cof{\nu}=\delta$ and $\nu\leq\kappa\leq\nu^+$. Then Lemma \ref{lemma:GoodWOSingular} and Lemma \ref{lemma:GoodWOSuccSingular} show that the set $\{U\}$ is $\Sigma_1(\kappa)$-definable. Since the restriction of the canonical wellordering of $\LL[U]$ to $\POT{\kappa}$ is a good $\Sigma_1(\kappa,U)$-wellordering of $\POT{\kappa}$, we can conclude that there is a good $\Sigma_1(\kappa)$-wellordering of $\POT{\kappa}$. 
\end{proof}

The above arguments also allow us to prove the following non-definability result for predecessors.

\begin{corollary}
 The following statements hold in the setting of Theorem \ref{theorem:Main}: 
 \begin{enumerate}
  \item If $\nu$ is a cardinal with $\cof{\nu}=\delta$, then the set $\{\delta\}$ is both $\Sigma_1(\nu)$ and $\Sigma_1(\nu^+)$-definable. 
  
  \item If $\mu\leq\nu$ are cardinals with $\cof{\mu}=\cof{\nu}=\delta$, then the sets $\{\mu\}$ and $\{\mu^+\}$ are not $\Sigma_1(\nu^{++})$-definable. 
 \end{enumerate}
\end{corollary}

\begin{proof}
 (i) This statement follows directly from Lemma \ref{lemma:GoodWOSingular} and Lemma \ref{lemma:GoodWOSuccSingular}. 
 
 (ii) Assume, towards a contradiction, that there are cardinals $\mu\leq\nu$ with the property that $\cof{\mu}=\cof{\nu}=\delta$ and either the set $\{\mu\}$ or the set $\{\mu^+\}$ is $\Sigma_1(\nu^{++})$-definable. Then the above lemmas show that the set $\{U\}$ is $\Sigma_1(\nu^{++})$-definable and this implies that the restriction of the canonical wellordering of $\LL[U]$ to $\POT{\nu^{++}}$ is a good $\Sigma_1(\nu^{++})$-wellordering. This contradicts Corollary \ref{corollary:CharacterizeAllGoodWOinLU}. 
\end{proof}


\section{Models of the form $\LL[U_0,U_1]$}\label{section:TwoMeasurables}

In this section, we study the provable restrictions that the existence of two measurable cardinals imposes on the existence of good $\Sigma_1(\kappa)$-wellorders. Throughout this section, we work in the setting of Theorem \ref{theorem:TwoMeasurables}: \emph{there are measurable cardinals $\delta_0<\delta_1$ and $\VV=\LL[U_0,U_1]$ holds, where $U_0$ is a normal ultrafilter on $\delta_0$ and $U_1$ is a normal ultrafilter in $\delta_1$}.

We start by proving the analog of Lemma \ref{lemma:GoodWOSingular} for two measurable cardinals.

\begin{lemma}\label{lemma:TwoMeasurablesSingluarCard}
 In the setting of Theorem \ref{theorem:TwoMeasurables}, if $\delta_0<\nu<\delta_1$ is a cardinal with $\cof{\nu}=\delta_0$, then the set $\{U_0\}$ is $\Sigma_1(\nu)$-definable. 
\end{lemma}

\begin{proof}
 By {\cite[Theorem 2.7]{MR926749}}, there is a finite fragment $\sf F$ of $\ZFC$ such that  the statement \anf{\emph{$M$ is a $D$-mouse}} is absolute between $\VV$ and transitive models of $\sf F$ containing $\kappa$ and $M$. This shows that there is a $\Sigma_1$-formula $\Phi(v_0,\ldots,v_5)$ with the property that the statement $\Phi(M,\varepsilon_0,\varepsilon_1,F_0,F_1,\nu)$ is equivalent to the conjunction of the following statements: 
 \begin{enumerate}
  \item $M$ is a $(\{\varepsilon_0\}\times F_0)$-mouse with $\nu,F_0,F_1\in\betrag{M}$ and $F_M=\{\varepsilon_1\}\times F_1$. 

  \item $\betrag{M}$ is a model of $\ZFC^-$. 
  
  \item $\varepsilon_0<\nu<\varepsilon_1$ and  $\lp{M}$ contains a strictly increasing cofinal function from $\varepsilon_0$ to $\nu$. 
 \end{enumerate}
  Then our assumptions imply that $\Phi(M,\delta_0,\delta_1,U_0,U_1,\nu)$ holds for some set $M$.

 \begin{claim*}
  If $\Phi(M,\varepsilon_0,\varepsilon_1,F_0,F_1,\nu)$ holds and $D=\{\varepsilon_0\}\times F_0$, then $\HHH{\nu}^{\KK[D]}\subseteq\betrag{M}$.
 \end{claim*}

 \begin{proof}[Proof of the Claim]
  Pick $x\in\HHH{\nu}^{\KK[D]}$. Since $\nu$ is a strong limit cardinal, the definition of $\KK[D]$ yields a $D$-mouse $N$ such that $\betrag{N}$ has cardinality less than $\nu$ and $\lp{N}$ contains a surjection $\map{s}{\xi}{\tc{\{x\}}}$ with $\xi<\nu$. By {\cite[Theorem 2.12]{MR926749}}, there is a simple $E$ and iterated ultrapowers $\map{j_0}{M}{M_*}$ and $\map{j_1}{N}{N_*}$ such that $\betrag{M_*}=\JJ_{\alpha(M_*)}[D,E]$ and $\betrag{N_*}=\JJ_{\alpha(N_*)}[D,E]$. Moreover, {\cite[Lemma 2.4]{MR926749}} implies that the corresponding embeddings $\map{j_0}{\langle\betrag{M},\in,D,F_M\rangle}{\langle\betrag{M_*},\in,D,E\rangle}$ and $\map{j_1}{\langle\betrag{N},\in,D,F_N\rangle}{\langle\betrag{N_*},\in,D,E\rangle}$ are both $\Sigma_1$-elementary.  Hence we can find $\zeta\in\betrag{N}$ with $\zeta=\min{(\meas{N})}$ and $j_0(\varepsilon_1)=j_1(\zeta)$. Since we have $$\zeta ~ = ~ \min{(\meas{N})} ~ < ~ \nu ~ < ~ \varepsilon_1 ~ \leq ~ j_0(\varepsilon_1) ~ = ~ j_1(\zeta)$$ and the $\GCH$ in $\LL[U_0,U_1]$ implies that $\nu$ is a strong limit cardinal, we know that $\nu$ appears as an image of $\zeta$ in a model in the sequence given by the above iterated ultrapower of $N$. In particular, $\nu$ is regular in $\betrag{N_*}$. Since $\lp{M}$ contains a strictly increasing cofinal function from $\varepsilon_0$ to $\nu$ and {\cite[Lemma 2.4]{MR926749}} implies that $\lp{M_*}$ also contains such a function, we can conclude that $\alpha(N_*)<\alpha(M_*)$. In this situation, we can conclude that $s\in\betrag{M}$, because {\cite[Lemma 2.4]{MR926749}} shows that $j_1\restriction\lp{N}=\id_{\lp{N}}$ and $\lp{N}\subseteq\HHH{\nu}^{\betrag{M_*}}=\HHH{\nu}^{\betrag{M}}$. 
\end{proof}

 \begin{claim*}
  If $\Phi(M,\varepsilon_0,\varepsilon_1,F_0,F_1,\nu)$ holds and $D=\{\varepsilon_0\}\times F_0$, then $\varepsilon_0$ is a measurable cardinal in $\KK[D]$,  $F_0$ is a normal ultrafilter on $\varepsilon_0$ in $\KK[D]$ and $\cof{\nu}^{\KK[D]}=\varepsilon_0$.
 \end{claim*}

 \begin{proof}[Proof of the Claim]
  Our assumptions imply that the first two statements hold in $\HHH{\nu}^{\betrag{M}}$ and therefore the above claim shows that they also hold in $\HHH{\nu}^{\KK[D]}$ and hence also in $\KK[D]$.  Now, assume, towards a contradiction, that $\cof{\nu}^{\KK[D]}<\varepsilon_0$. Since {\cite[Theorem 3.2]{MR926749}} shows that $\KK[D]$ is a model of $\ZFC$ and $\lp{M}\subseteq\KK[D]$, this assumption implies that $\varepsilon_0$ is singular in $\KK[D]$ and this contradicts the above conclusions.  
 \end{proof}

 \begin{claim*} 
  If $\Phi(M,\varepsilon_0,\varepsilon_1,F_0,F_1,\nu)$ holds, then $\delta_0=\varepsilon_0$ and $F_0=U_0$. 
\end{claim*} 

\begin{proof}[Proof of the Claim]
 Set $D=\{\delta_0\}\times U_0$ and $E=\{\varepsilon_0\}\times F_0$. First, assume, towards a contradiction, that $\delta_0\neq\varepsilon_0$. Then $\delta_0<\varepsilon_0$, because $\delta_0=\cof{\nu}\leq\cof{\nu}^{\KK[E]}=\varepsilon_0$. Since $\VV=\LL[U_0,U_1]$, we can apply {\cite[Theorem 2.14]{MR926749}} to show that the assumptions of {\cite[Theorem 3.16]{MR926749}} are satisfied. Since both $D$ and $E$ code measure sequences of length $1$, this result shows that $\langle\KK[E],\in,F_0\rangle$ is an iterate of $\langle\KK[D],\in,U_0\rangle$ given by an iteration of length $0<\tau<\nu$. Let $\map{j}{\KK[D]}{\KK[E]}$ denote the corresponding elementary embedding $j(\delta_0)=\varepsilon_0$ and $j(U_0)=F_0$. Then {\cite[Corollary 19.7]{MR1994835}} implies that $j(\mu)<\nu$ holds for all $\mu<\nu$ and there is a continuous, cofinal and strictly increasing map $\map{c}{\delta_0}{\nu}$ in $\KK[D]$ with the property that $j(c(\gamma))=c(\gamma)$ holds for every $\gamma<\delta_0$. But then $c=j(c)\restriction\delta_0\in\KK[E]$ and $\cof{\nu}^{\KK[E]}\leq\delta_0<\varepsilon_0$, contradicting the above claim. 
 
 The above computations show that $\delta_0=\varepsilon_0$. In this situation, we can apply {\cite[Theorem 3.14]{MR926749}} to conclude that $\KK[D]=\KK[E]$ and $F_0=U_0$.  
\end{proof}

The above claim shows that the filter $U_0$ is the unique set $F_0$ such that there are $M$, $\varepsilon_0$, $\varepsilon_1$ and $F_1$ with the property that $\Phi(M,\varepsilon_0,\varepsilon_1,F_0,F_1,\nu)$ holds. This allows us to conclude that the set $\{U_0\}$ is definable by a $\Sigma_1$-formula with parameter $\nu$. 
\end{proof}

 Next, we also generalize Lemma \ref{lemma:GoodWOSuccSingular} to the two-measurable-cardinals setting.

\begin{lemma}\label{lemma:TwoMeasurablesSuccSingluarCard}
  In the setting of Theorem \ref{theorem:TwoMeasurables}, if $\delta_0<\mu<\delta_1$ is a cardinal with $\cof{\mu}=\delta_0$, then the set $\{U_0\}$ is $\Sigma_1(\mu^+)$-definable. 
\end{lemma}

\begin{proof}
 Set $\kappa=\mu^+$ and $D=\{\delta_0\}\times U_0$. Then our assumptions imply that $\KK[D]$ is a model of $\ZFC$, $U_0\in\KK[D]$ is a normal ultrafilter on $\delta_0$ and  $\HHH{\delta_1}\subseteq\KK[D]$. In particular, $\cof{\mu}^{\KK[D]}=\delta_0$ and $\kappa=(\mu^+)^{\KK[D]}$. Let $$\langle\seq{\langle N_\gamma,\in,W_\gamma\rangle}{\gamma<\kappa},\seq{\map{j_{\beta,\gamma}}{N_\beta}{N_\gamma}}{\beta\leq\gamma<\kappa}\rangle$$ denote the system of iterated ultrapowers of $\langle\KK[D],\in,U_0\rangle$. Since $\mu$ is a strong limit cardinal greater than $\delta_0$, we have $j_{0,1}(\delta_0)<\mu$ and hence $j_{0,\gamma}(\delta_0)<j_{1,\gamma}(\mu)$ for all $0<\gamma<\kappa$. The following statement is shown in the proof of Lemma \ref{lemma:GoodWOSuccSingular}.

 \begin{claim*}
  $\cof{j_{1,\gamma}(\mu)}^{N_\gamma}=\delta_0$ and $\kappa=(j_{1,\gamma}(\mu)^+)^{N_\gamma}=j_{1,\gamma}(\kappa)$ for all $0<\gamma<\kappa$. \qed 
 \end{claim*}

 Using {\cite[Theorem 2.7]{MR926749}}, we find a $\Sigma_1$-formula $\Psi(v_0,\ldots,v_6)$ with the property that $\Psi(M,\varepsilon_0,\varepsilon_1,F_0,F_1,\nu,\kappa)$ is equivalent to the conjunction of the following statements: 
 \begin{enumerate}
    \item $M$ is a $(\{\varepsilon_0\}\times F_0)$-mouse with $\nu,\kappa,F_0,F_1\in\betrag{M}$ and $F_M=\{\varepsilon_1\}\times F_1$. 

  \item $\betrag{M}$ is a model of $\ZFC^-$. 
  
  \item $\varepsilon_0<\nu<\kappa<\varepsilon_1$, $\kappa=(\nu^+)^{\betrag{M}}$  and  $\lp{M}$ contains a strictly increasing cofinal function from $\varepsilon_0$ to $\nu$.  
 \end{enumerate}
 Note that our assumptions imply that the statement $\Psi(M,\delta_0,\delta_1,U_0,U_1,\mu,\kappa)$ holds for some set $M$.

 \begin{claim*}
  If $\Psi(M,\varepsilon_0,\varepsilon_1,F_0,F_1,\nu,\kappa)$ holds and $E=\{\varepsilon_0\}\times F_0$, then $\HHH{\kappa}^{\KK[E]}\subseteq\betrag{M}$.
 \end{claim*}

 \begin{proof}[Proof of the Claim]
  Fix $x\in\HHH{\kappa}^{\KK[E]}$. Since $\kappa$ is a cardinal, there is an $E$-mouse $N$ with the property that $\betrag{N}$ has cardinality less than $\kappa$ and $\lp{N}$ contains a surjection $\map{s}{\xi}{\tc{\{x\}}}$ with $\xi<\kappa$. Then {\cite[Theorem 2.12]{MR926749}} yields a simple set $F$ and iterated ultrapowers $\map{j_0}{M}{M_*}$ and $\map{j_1}{N}{N_*}$ such that  $\betrag{M_*}=\JJ_{\alpha(M_*)}[E,F]$ and $\betrag{N_*}=\JJ_{\alpha(N_*)}[E,F]$. Using {\cite[Lemma 2.4]{MR926749}}, we find $\zeta\in\betrag{N}$ with $\zeta=\min{(\meas{N})}$ and $j_0(\varepsilon_1)=j_1(\zeta)$.  Since $\kappa$ is a regular cardinal, $\betrag{N}$ has cardinality less than $\kappa$ and $$\zeta ~ = ~ \min{(\meas{N})} ~ < ~ \kappa ~ < ~ \varepsilon_1 ~ \leq ~  j_0(\varepsilon_1) ~ = ~ j_1(\zeta), $$ we know that $\kappa$ appears as an image of $\zeta$ in a model in the sequence given by the above iterated ultrapower of $N$ and this implies that $\kappa$ is inaccessible in $\betrag{N_*}$. Since $\kappa=j_0(\kappa)$ is a successor cardinal in $\betrag{M_*}$, we can conclude that $\alpha(N_*)<\alpha(M_*)$. In particular, we know that  $s=j_1(s)\in\HHH{\kappa}^{\betrag{M_*}}=\HHH{\kappa}^{\betrag{M}}$.  
 \end{proof}

 Analogous to the proof of Lemma \ref{lemma:TwoMeasurablesSingluarCard}, we can derive the following statement from the above claim.

 \begin{claim*}
  If $\Psi(M,\varepsilon_0,\varepsilon_1,F_0,F_1,\nu,\kappa)$ holds and $E=\{\varepsilon_0\}\times F_0$, then $\varepsilon_0$ is a measurable cardinal in $\KK[E]$,  $F_0$ is a normal ultrafilter on $\varepsilon_0$ in $\KK[E]$, $\cof{\nu}^{\KK[E]}=\varepsilon_0$ and $\kappa=(\nu^+)^{\KK[E]}$. \qed
 \end{claim*}

 \begin{claim*} 
  If $\Psi(M,\varepsilon_0,\varepsilon_1,F_0,F_1,\nu,\kappa)$ holds, then $\delta_0=\varepsilon_0$ and $F_0=U_0$. 
\end{claim*} 

\begin{proof}[Proof of the Claim]
 Define $E=\{\varepsilon_0\}\times F_0$. Since $\kappa=(\nu^+)^{\KK[E]}$, we have $\mu\leq\nu$. 
 
Assume that $\delta_0\neq\varepsilon_0$. Then our assumptions and {\cite[Theorem 3.16]{MR926749}} show that either $\langle\KK[D],\in,U_0\rangle$ is an iterate of $\langle\KK[E],\in,F_0\rangle$ or $\langle\KK[E],\in,F_0\rangle$ is an iterate of $\langle\KK[D],\in,U_0\rangle$.  If $\mu=\nu$, then $\delta_0\neq\varepsilon_0$ implies  $\cof{\mu}^{\KK[D]}=\delta_0<\varepsilon_0=\cof{\mu}^{\KK[E]}$ and the second option holds. In the other case, if $\mu<\nu$, then $\nu$ is not a cardinal in $\KK[D]$ and the second option also holds. Hence we can find $0<\gamma\leq\varepsilon_0<\kappa$ with $\KK[E]=N_\gamma$, $\varepsilon_0=j_{0,\gamma}(\delta_0)$ and $F_0=W_\gamma$. By the above claims, we know that $j_{1,\gamma}(\mu)$ is a cardinal of cofinality $\delta_0$ in $\KK[E]$, $\nu$ is a cardinal of cofinality $\varepsilon_0$ in $\KK[E]$ and  $(\nu^+)^{\KK[E]}=\kappa=(j_{1,\gamma}(\mu)^+)^{\KK[E]}$. This yields $\nu=j_{1,\gamma}(\mu)$ and $\delta_0=\varepsilon_0$, a contradiction.

 The above computations show that $\delta_0=\varepsilon_0$. As above, this allows us to use  {\cite[Theorem 3.14]{MR926749}} to show that $\KK[D]=\KK[E]$ and $F_0=U_0$.  
\end{proof}

Using the above claims, we can conclude that  the filter $U_0$ is the unique set $F_0$ such that there are $M$, $\varepsilon_0$, $\varepsilon_1$,  $F_1$ and $\nu$ such that $\Psi(M,\varepsilon_0,\varepsilon_1,F_0,F_1,\nu,\kappa)$ holds and this implies that the set $\{U_0\}$ is definable by a $\Sigma_1$-formula with parameter $\kappa$. 
\end{proof}

\begin{proof}[Proof of Theorem \ref{theorem:TwoMeasurables}]
  Assume that there are measurable cardinals $\delta_0<\delta_1$ and $\VV=\LL[U_0,U_1]$ holds, where $U_0$ is a normal ultrafilter on $\delta_0$ and $U_1$ is a normal ultrafilter in $\delta_1$. Let $\kappa$ be an infinite cardinal. First, assume that there is a good $\Sigma_1(\kappa)$-wellordering of $\POT{\kappa}$. Then Lemma \ref{lemma:MeasurableWOReals} implies that $\kappa$ is uncountable and, since the $\GCH$ holds in $\LL[U_0,U_1]$,  Corollary \ref{corollary:SimplificationGCHNonMeasurable} implies that either $\omega<\kappa<\delta_0$ or there is a cardinal $\nu<\delta_1$ with $\cof{\nu}=\delta_0$ and $\nu\leq\kappa\leq\nu^+$.  Next, if $\omega<\kappa<\delta_0$, then Corollary  \ref{corollary:LowWOTwoMeasurables} shows that there is a good $\Sigma_1(\kappa)$-wellordering of $\POT{\kappa}$. Finally, assume that there is a cardinal $\nu<\delta_1$ with $\cof{\nu}=\delta_0$ and $\nu\leq\kappa\leq\nu^+$. Set $D=\{\delta_0\}\times U_0$. 
 Since $\HHH{\delta_1}\subseteq\Ult{\VV}{U_1}$, we have $\HHH{\delta_1}\subseteq\KK[D]$ and Lemma \ref{lemma:GoodWOinK} shows that there is a good $\Sigma_1(\kappa,U_0)$-wellordering of $\POT{\kappa}$. Since we can apply Lemma \ref{lemma:TwoMeasurablesSingluarCard} or Lemma \ref{lemma:TwoMeasurablesSuccSingluarCard} to see that the set $\{U_0\}$ is $\Sigma_1(\kappa)$-definable, we can conclude that there is a good $\Sigma_1(\kappa)$-wellordering of $\POT{\kappa}$ in this case. 
\end{proof}


\section{Open Questions}\label{section:questions}

We conclude this paper by stating some question raised by the above results.

 Our main results only talk about models of set theory in which the $\GCH$ holds. It is not clear whether a failure of the $\GCH$ together with the existence of a measurable cardinal might impose more restrictions on the existence of good $\Sigma_1$-wellorderings than the ones given by the results of Section \ref{section:intro}. In particular, it is natural to consider the following question.

\begin{question}
 Is it consistent that there is a measurable cardinal $\delta$ with $2^\delta>\delta^+$ and there is a good $\Sigma_1(\delta^+)$-wellordering of $\POT{\delta^+}$?
\end{question}

Next, we consider slightly less simply definable wellorderings. The above proofs of the nonexistence of good $\Sigma_1$-wellorderings in $\LL[U]$  leave open the following question.

\begin{question}
 Assume that $\VV=\LL[U]$ for some normal ulterafilter $U$ on a measurable cardinal. If $\kappa$ is an uncountable cardinal, is there a $\Sigma_1(\kappa)$-definable wellordering of $\POT{\kappa}$?
\end{question}

The results of \cite{Lightface} mentioned in Section \ref{section:intro} show that certain large cardinals imply the nonexistence of good $\Sigma_1$-wellorderings at $\omega_1$ and the large cardinals themselves. It is not know whether similar implications also hold for other uncountable regular cardinals.

\begin{question}
 Is the existence of a good $\Sigma_1(\omega_2)$-wellordering of $\POT{\omega_2}$ consistent with the existence of a supercompact cardinal? 
\end{question}


\bibliographystyle{plain}
 \bibliography{references}


\end{document}